\documentclass[12pt]{amsart}

\usepackage{a4wide}

\usepackage{amssymb,amsmath}

\newtheorem{theorem}{Theorem}
\newtheorem{lemma}[theorem]{Lemma}
\newtheorem{corollary}[theorem]{Corollary}

\begin{document}
\title{Proper isometric actions}

\author{Jos\'{e} Carlos D\'\i{}az-Ramos}

\thanks{The author has been supported by a Marie-Curie
Intra-European Fellowship (MEIF-CT-2006-038754) and by project
PGIDIT06PXIB207054PR (Spain).}

\maketitle

%\begin{center}{Draft version, \today}\end{center}

In this short note we use a theorem by Helgason to give an easy
proof of two results on proper isometric actions
(theorems~\ref{thProperClosed} and \ref{thClosedOrbits}).

Throughout this paper we assume that $M$ is a connected and complete
Riemannian manifold. We denote by $I(M)$ the full isometry group of
$M$. It is a well known fact that $I(M)$ is a Lie group \cite{MS39}.

Let $G$ be a Lie group. An action $\varphi:G\times M\to M$ is said
to be an \emph{isometric action} if for each $g\in G$ the map
$\varphi_g:M\to M$, $p\mapsto g(p)$, is an isometry of $M$. It is
customary to restrict to effective actions so that $G$ can be
considered a subgroup of $I(M)$ (not necessarily closed).

Our basic tool is the following theorem, which was proved in
\cite[p. 167]{H62} (see also \cite{Y77}).

\begin{theorem}\label{thHelgason}
Let $M$ be a complete, connected Riemannian manifold and let
$\{g_n\}\subset I(M)$ be a sequence such that $\{g_n(p)\}$ is
bounded in $M$ for some $p\in M$. Then, there exists an isometry
$g\in I(M)$ and a subsequence $\{g_{n_k}\}$ of $\{g_n\}$ such that
$g_{n_k}\to g$.
\end{theorem}

An action of a Lie group $G$ on a manifold $M$ is called
\emph{proper} \cite{P61} if one (and hence all) of the following
conditions is satisfied:
\begin{itemize}
\item[(i)] The map $G\times M\to M\times M$, $(g,p)\mapsto (g(p),p)$
is proper.

\item[(ii)] Given compact subsets $K$ and $L$ of $M$, the set
$\{g\in G:g(K)\cap L\neq\emptyset\}$ is compact.

\item[(iii)] For any two points $p,q\in M$, there exist
open neighborhoods $U_p$ and $V_q$ of $p$ and $q$ respectively such
that $\{g\in G:g(U_p)\cap V_q\neq\emptyset\}$ is relatively compact.

\item[(iv)] For any sequences $\{g_n\}\subset G$ and
$\{p_n\}\subset M$, if $g_n(p_n)\to q$ and $p_n\to p$ then
$\{g_n\}$ has a convergent subsequence.
\end{itemize}

For isometric actions condition (iv) can be simplified:

\begin{lemma}
Let $G\subset I(M)$. Then $G$ acts properly on $M$ if and only if
for any sequence $\{g_n\}\subset G$ such that $\{g_n(p)\}$ is
bounded for some $p\in M$, there exists a convergent subsequence of
$\{g_n\}$ in $G$.
\end{lemma}

\begin{proof}
The ``if'' part is obvious because any bounded sequence of points of
$M$ has a convergent subsequence. Conversely, let $\{g_n\}\subset G$
and $\{p_n\}\subset M$ be sequences such that $g_n(p_n)\to q$ and
$p_n\to p$. Since $0\leq d(g_n(p),q)\leq
d(g_n(p),g_n(p_n))+d(g_n(p_n),q)=d(p,p_n)+d(g_n(p_n),q)$ we have
$g_n(p)\to  q$. In particular, $\{g_n(p)\}$ is bounded. By
hypothesis there exists a subsequence $\{g_{n_k}\}$ of $\{g_n\}$
such that $g_{n_k}\to g\in G$. Hence $G$ acts properly on $M$.
\end{proof}

\begin{corollary}\label{thProper}
A subgroup $G\subset I(M)$ acts properly on $M$ if and only if for
any sequence $\{g_n\}\subset G$ and any $p\in M$, $g_n(p)\to q$
implies that $\{g_n\}$ has a convergent subsequence in $G$.
\end{corollary}

It is well known that a closed subgroup of the isometry group of a
Riemannian manifold acts properly on that manifold. The converse
fact seems to be a folklore theorem, but the author has not found a
proof in the literature. We give a simple proof using
Theorem~\ref{thHelgason}.

\begin{theorem}\label{thProperClosed}
Let $G\subset I(M)$ be a subgroup. Then, $G$ acts properly on $M$ if
and only if $G$ is closed in $I(M)$.
\end{theorem}

\begin{proof}
First assume that $G$ acts properly on $M$. Let $\{g_n\}$ be a
sequence with $g_n\to g\in I(M)$. Let $p\in M$. Then, $g_n(p)\to
g(p)$, and so by Corollary \ref{thProper} there exists a subsequence
$\{g_{n_k}\}$ such that $g_{n_k}\to h\in G$. The uniqueness of limit
implies $g=h\in G$, so $G$ is closed.

Conversely, assume that $G$ is closed. Let $\{g_n\}\subset G$ be a
sequence with $g_n(p)\to q$. By Theorem \ref{thHelgason}, there
exists a subsequence $\{g_{n_k}\}$ such that $g_{n_k}\to g\in I(M)$.
Since $G$ is closed it follows that $g\in G$ and thus $G$ acts
properly by Corollary \ref{thProper}.
\end{proof}

The following result states that, up to orbit equivalence, isometric
actions with closed orbits correspond to proper isometric actions.

\begin{theorem}\label{thClosedOrbits}
Let $G\subset I(M)$ be a subgroup. Then, the orbits of $G$ are
closed if and only if the action of $G$ is orbit equivalent to the
action of the closure of $G$ in $I(M)$.
\end{theorem}

\begin{proof}
Let $\bar{G}$ denote the closure of $G$ in $I(M)$.

Assume that the orbits of $G$ are closed. Let $p\in M$. Obviously,
$G\cdot p\subset \bar{G}\cdot p$. Let $q\in\bar{G}\cdot p$ and write
$q=g(p)$ with $g\in\bar{G}$. Take $\{g_n\}\subset G$ such that
$g_n\to g$. Then, $g_n(p)\to g(p)\in G\cdot p$ because $G\cdot p$ is
closed. By uniqueness of limit we have $q=g(p)\in G\cdot p$.

Conversely, assume that the action of $G$ is orbit equivalent to the
action of $\bar{G}$. Let $p\in M$ and take $\{g_n\}\subset G$ with
$g_n(p)\to q\in M$. By Theorem \ref{thHelgason} there is a
subsequence $\{g_{n_k}\}$ such that $g_{n_k}\to g\in\bar{G}$. Hence
$g_{n_k}(p)\to g(p)$. By uniqueness of limit $q=g(p)\in\bar{G}\cdot
p=G\cdot p$ so $G\cdot p$ is closed.
\end{proof}

\begin{corollary}
The orbits of an isometric action are closed if and only if the
action is orbit equivalent to a proper isometric action.
\end{corollary}

\noindent\textbf{Acknowledgement}: The author would like to thank
Prof. J. Szenthe for his useful comments and observations regarding
the contents of this paper.

%%%%%%%%%%%%%%%%%%%%%%%%%% Bibliography %%%%%%%%%%%%%%%%%%%%%%%%%%%

%\bigskip
\noindent{\sc J. C. D\'{\i}az-Ramos}: Department of Mathematics,
University College Cork, Ireland.\\
E-mail: jc.diazramos@ucc.ie

\end{document}